\pdfoutput=1
\documentclass{article}
\usepackage{amsmath,amsthm,amssymb,stmaryrd,braket,mathtools,mathdots,enumitem}
\usepackage[theoremfont]{newtxtext}
\usepackage{newtxmath}
\renewcommand{\forall}{\forallAlt}
\renewcommand{\exists}{\existsAlt}

\usepackage{hyperref}
\usepackage[initials]{amsrefs}

\theoremstyle{plain}
\newtheorem{theorem}{Theorem}
\newtheorem{proposition}[theorem]{Proposition}
\newtheorem{corollary}[theorem]{Corollary}
\newtheorem{lemma}[theorem]{Lemma}
\theoremstyle{definition}
\newtheorem{definition}[theorem]{Definition}
\theoremstyle{remark}
\newtheorem*{remark*}{Remark}

\let\origqedsymbol\qedsymbol
\newcommand{\N}{\mathbb{N}}
\newcommand{\RCA}{\mathsf{RCA}_0}
\newcommand{\WKL}{\mathsf{WKL}_0}
\DeclareMathOperator{\id}{id}
\newcommand{\WPH}{\mathrm{WPH}}
\newcommand{\MDL}{\mathrm{MDL}}
\newcommand{\RPH}{\mathrm{RPH}}
\newcommand{\DL}{\mathrm{DL}}
\newcommand{\PH}{\mathrm{PH}}
\newcommand{\EFA}{\mathsf{EFA}}
\newcommand{\PRA}{\mathsf{PRA}}
\newcommand{\ls}{\Sigma}
\newcommand{\lp}{\Pi}
\newcommand{\ld}{\Delta}
\newcommand{\ra}{\rightarrow}
\newcommand{\lr}{\leftrightarrow}
\newcommand{\Ra}{\Rightarrow}
\DeclareMathOperator{\len}{lh}
\newcommand{\conc}{\mathbin{\text{\^{}}}}

\DeclarePairedDelimiter\abs{\lvert}{\rvert}
\DeclarePairedDelimiter\norm{\lvert}{\rvert_{\infty}}

\makeatletter
\newcommand{\dotminus}{\mathbin{\text{\@dotminus}}}
\newcommand{\@dotminus}{%
	\ooalign{\hidewidth\raise1ex\hbox{.}\hidewidth\cr$\m@th-$\cr}%
}
\makeatother

\title{Dickson's Lemma and Weak Ramsey Theory}
\author{Yasuhiko Omata\thanks{Supported in part by JSPS Kakenhi 15H03634} \qquad Florian Pelupessy\thanks{Supported by the Japan Society for the Promotion of Science} \\ Mathematical Institute, Tohoku University}
\date{October 3, 2017} 
\begin{document}
	\maketitle
	
	\begin{abstract}
		We explore the connections between Dickson's lemma and weak Ramsey theory.
		We show that a weak version of the Paris--Harrington principle for pairs in $c$ colors and miniaturized Dickson's lemma for $c$-tuples are equivalent over $\mathsf{RCA}_0^{\ast}$.
		Furthermore, we look at a cascade of consequences for several variants of weak Ramsey's theorem.
	\end{abstract}
	
	\section{Introduction}
	Dickson's lemma, originally used in algebra, in particular for showing Hilbert's basis theorem~\cite{G}, is nowadays commonly used in termination proofs in computer science~\cite{Setal}.
	The weak Paris--Harrington principle for pairs was originally used as an easy intermediate version in showing lower bounds for the Paris--Harrington principle for pairs~\cite{E-M}.
	We provide simple constructions which show that the weak Paris--Harrington principle and miniaturized Dickson's lemma are equivalent over $\RCA^{\ast}$, the base theory weaker than $\RCA$.
	Additionally our construction provides an explicit formula for weak Ramsey numbers and tight upper bounds for the weak Paris--Harrington principle derived from those for Dickson's lemma.
	
	$\N$ denotes the set of nonnegative integers.
	We define some notations for colorings.
	For $a,R,c\in\N$, $[a,R]$ and $[a,R]^2$ denote the sets $\set{n\in\N:a\leq n\leq R}$ and $\set{(n,m)\in\N^2:a\leq n<m\leq R}$ respectively, and $c$ is identified with the set $[0,c-1]=\set{n\in\N:n<c}$.
	Given a map $C\colon [a,R]^2\ra c$ (called {\em coloring}), we say that a set $H\subseteq [a,R]$ is {\em $C$-homogeneous} if $C$ is constant on $[H]^2=\set{(n,m)\in H^2:n<m}$.
	Similarly, we say that a set $H=\set{h_0<h_1<\dotsb}\subseteq [a,R]$ is {\em $C$-weakly homogeneous} if $C(h_i,h_{i+1})=C(h_{i+1},h_{i+2})$ holds for all $h_i,h_{i+1},h_{i+2}\in H$.
	Weakly homogeneous sets are sometimes called {\em adjacent homogeneous} or {\em path homogeneous}.
	
	\begin{definition}[the weak Paris--Harrington principle]
		For $f\colon\N\ra\N$ and $c,a,R\in\N$, let $\WPH^f_c(a,R)$ be the statement that for every coloring $C\colon[a,R]^2\ra c$ there exists a $C$-weakly homogeneous set $H\subseteq[a,R]$ with $\abs*{H}>f(\min H)$.
		\emph{The weak Paris--Harrington principle for pairs and $c$ colors with parameter $f$}, denoted $\WPH^f_c$, states that for every $a$ there exists $R$ such that $\WPH^f_c(a,R)$ holds.
	\end{definition}
	
	We also define the notations for tuples.
	For $c$-tuples $\overline{m}=\left(m_0,\dotsc,m_{c-1}\right),\overline{n}=\left(n_0,\dotsc,n_{c-1}\right)\in\N^c$, write $\overline{m}\leq\overline{n}$ if and only if $\forall k<c\left(m_k\leq n_k\right)$, and $\norm*{\overline{m}}=\max_{k<c}\set{m_k}$.
	
	\begin{definition}[miniaturized Dickson's lemma]
		For $f\colon\N\ra\N$ and $c,a,D\in\N$, let $\MDL^f_c(a,D)$ be the statement that for every sequence $\overline{m}_0,\dotsc,\overline{m}_D\in\N^c$ with $\norm*{\overline{m}_i}<f(a+i)$ there exists $i<j\leq D$ such that $\overline{m}_i\leq\overline{m}_j$.
		\emph{Miniaturized Dickson's lemma for $f$ for $c$-tuples}, denoted $\MDL^f_c$, states that for every $a$ there exists $D$ such that $\MDL^f_c(a,D)$ holds.
	\end{definition}
	
	Our original intent was to provide direct proof of equivalence of Dickson's lemma ($\DL$) and $\forall c\forall f\WPH^f_c$ (\textup{Corollary~\ref{DL}}) and equivalence of $\WPH^{\id}_c$ and $\MDL^{\id}_c$ (\textup{Corollary~\ref{MDL}}).
	With some work, this could already be shown using proofs of equivalences of 
	\begin{itemize}[label={--}]
		\item
		$\forall c\PH^{\id}$ and $1\mathchar`-\mathrm{Con}(\mathsf{I}\ls_1)$ (\cite{K-S}),
		\item
		$\forall c\forall f\PH^f$ and $\mathrm{WO}(\omega^{\omega})$ (\cite{K-Y}),
		\item
		$\DL$ and $\mathrm{WO}(\omega^{\omega})$ (\cite{S}).
	\end{itemize}
	However, this method, from the previous literature, gives us the weak implication $\mathrm{WO}(\omega^{c+4})\ra\forall f\WPH^f_c$, while our work shows the level-by-level equivalence between $\forall f\WPH^f_c$ and $\DL_c$ (which is also equivalent to $\mathrm{WO}(\omega^c)$) in \textup{Corollary~\ref{DL}}.
	
	Our method, additionally, gives a similar sharpening of complexity bounds, stated in \textup{Corollaries~\ref{eq},~\ref{func}}, and the explicit expression in \textup{Theorem~\ref{wr}} for the weak Ramsey numbers.
	
	Finally, we look at the consequences, for the bounds of weak Ramsey numbers in higher dimensions (\textup{Section~\ref{secwr}}), and the phase transitions which follow from these bounds (\textup{Section~\ref{secpt}}).
	
	For examinations of weak Ramsey's theorem and its relation to termination we refer the reader to~\cite{S-Y}.
	
	\section{Base theory $\RCA^{\ast}$}
	Most of the results in this paper can be established within $\RCA^{\ast}$.
	
	\begin{definition}[$\RCA^{\ast}$]
		$\RCA^{\ast}$ is the subsystem of second order arithmetic, whose language additionaly contains binary function symbol $\mathtt{exp}$, consists of the following axioms:
		\begin{enumerate}
			\item
			basic axioms (see~\cite[\textup{Definition~I.2.4 (i)}]{sosoa});
			\item
			exponentiation axioms:
			\begin{align*}
			\mathop{\mathtt{exp}}(m,0)&=1,\\
			\mathop{\mathtt{exp}}(m,n+1)&=m\cdot\mathop{\mathtt{exp}}(m,n);
			\end{align*}
			\item
			induction scheme for all $\ls^0_0$ formulas which may contain $\mathop{\mathtt{exp}}$;
			\item
			comprehension scheme for all $\ld^0_1$ formulas which may contain $\mathop{\mathtt{exp}}$.
		\end{enumerate}
	\end{definition}
	
	$\mathop{\mathtt{exp}}(m,n)$ will be just denoted $m^n$.
	
	$\RCA^{\ast}$ is essentially $\EFA$ (elementary function arithmetic) plus $\ld^0_1$-comprehension.
	The relation between $\RCA^{\ast}$ and $\EFA$ is similar to the relation between $\RCA$ and $\PRA$ (primitive recursive arithmetic).
	$\RCA^{\ast}$ is $\lp^0_2$-conservative over $\EFA$, while $\RCA$ is $\lp^0_2$-conservative over $\PRA$.
	For more details about $\RCA^{\ast}$ and the conservativity results, see~\cite{S-S}.
	
	\begin{lemma}\label{bpr}
		$\RCA^{\ast}$ proves the closure under \emph{the bounded course of value primitive recursion}:
		For all functions $b\colon\N\times\N^k\ra\N$ and $g\colon\N\times\N\times\N^k\ra\N$, there exists the unique function $h\colon\N\times\N^k\ra\N$ satisfying
		\begin{equation*}
		h(n,\overline{m})=\min\set{b(n,\overline{m}),g(\left\langle h(0,\overline{m}),\dotsc,h(n-1,\overline{m})\right\rangle,n,\overline{m})}.
		\end{equation*}
	\end{lemma}
	
	\begin{proof}
		This proof is almost same as~\cite[\textup{Lemma~2.2}]{S-S}.
		
		Fix any $\overline{m}$.
		First, define the function $j(n)$ by the following primitive recursion:
		\begin{equation*}
		\begin{cases}
		j(0)&=0,\\
		j(n+1)&=
		\begin{cases}
		y+1&\text{if }b(n+1,\overline{m})\geq b(j(n),\overline{m}),\\
		j(n)&\text{otherwise}.
		\end{cases}
		\end{cases}
		\end{equation*}
		
		One can define the graph of $j$ by
		\begin{equation*}
		j(n)=y\lr\exists c<n^n\left(
		\begin{aligned}
		&(c)_0=0\land (c)_n=y\\
		&\land\forall i<n\left[\left((c)_{i+1}=n+1\land b(n+1)+1>b((c)_i)\right)\right.\\
		&\left.\lor\left((c)_{i+1}=(c)_i\land b(n+1)<b((c)_i)\right)\right]
		\end{aligned}
		\right)
		\end{equation*}
		using $\ld^0_1$-comprehension and $j$ is a function by $\ls^0_0$-induction.
		
		Since $b(j(n),\overline{m})=\max\set{b(n',\overline{m}):n'\leq n}$, the sequence
		\begin{equation*}
		\left\langle b(0,\overline{m}),b(1,\overline{m}),\dotsc,b(n-1,\overline{m})\right\rangle
		\end{equation*}
		is coded by some natural number less than $b(j(n),\overline{m})^n$.
		Then we can define $h$ in the same way by
		\begin{equation*}
		h(n,\overline{m})=y\lr\exists c<b(j(n))^n\left(
		\begin{aligned}
		&\len(c)=n\land (c)_n=y\\
		&\land\forall i\leq n\left[(c)_i=\min\set{b(i,\overline{m}),g(c\restriction i,i,\overline{m})}\right]
		\end{aligned}
		\right).
		\end{equation*}
		The uniqueness of $h$ is also proven by $\ld^0_1$-comprehension and $\ls^0_0$-induction.
	\end{proof}
	
	\textup{Lemma~\ref{bpr}} implies the following well-known result.
	
	\begin{corollary}
		$\RCA^{\ast}$ proves the existence of every elementary recursive function.
	\end{corollary}
	
	\section{Constructions}
	We provide the notions of bad colorings/sequences.
	They are counterexamples to $\WPH^f_c(a,R)$ and $\MDL^f_c(a,R)$ respectively.
	
	\begin{definition}[bad coloring]
		Given $a,c,R\in\N$ and $f\colon\N\ra\N$, a coloring $C\colon [a,R]^2\ra c$ is \textit{$f$-bad} if every $C$-weakly homogeneous set $H\subseteq [a,R]$ has size $\leq f(\min H)$.
	\end{definition}
	
	\begin{definition}[bad sequence]
		Let $a,c,D\in\N$ and $f\colon\N\ra\N$ be given.
		We say that a sequence $\overline{m}_0,\dotsc,\overline{m}_D\in\N^c$ is \emph{bad} if for all $i<j\leq D$, $\overline{m}_i\nleq\overline{m}_j$ holds.
		Also, we say that $\overline{m}_0,\dotsc,\overline{m}_D$ is \emph{$(a,f)$-bounded} if $\norm{\overline{m}_i}<f(a+i)$ for all $i\leq D$.
		We call $(a,f)$-bounded bad sequences \emph{$(a,f)$-bad}.
	\end{definition}
	
	Then $\WPH^f_c$ states that for every $a$ there exists $R$ such that there is no $f$-bad coloring $C\colon[a,R]^2\ra c$, and $\MDL^f_c$ states that for every $a$ there exists $D$ such that there is no $(a,f)$-bad sequence $\overline{m}_0,\dotsc,\overline{m}_D\in\N^c$.
	
	\begin{lemma}[$\RCA^{\ast}$]\label{lem}
		For every $f\colon\N\ra\N$ and $c,a,R,D\in\N$, the following hold:
		\upshape
		\begin{enumerate}[label=(\roman*)]
			\item\label{F.}
			Existence of an $f$-bad coloring $C\colon[a,R]^2\ra c$ implies existence of an $(a,f)$-bad sequence $\overline{m}_0,\dotsc,\overline{m}_{R-a}\in\N^c$.
			\item\label{O.}
			Existence of an $(a,f)$-bad sequence $\overline{m}_0,\dotsc,\overline{m}_D\in\N^c$ implies existence of an $f$-bad coloring $C\colon[a,a+D]^2\ra c$.
		\end{enumerate}
		The same holds for bad colorings $C\colon[a,\infty]^2\ra c$ and infinite $(a,f)$-bad sequences.
	\end{lemma}
	\begin{proof}[Proof of~\ref{F.}]
		Let $C\colon[a,R]^2\ra c$ be a given $f$-bad coloring.
		The idea of construction is to construct a sequence of $c$-tuples with the following properties:
		\begin{enumerate}
			\item
			If $C(a+j,a+i)=k$, then $\big(\overline{m}_j\big)_k>\left(\overline{m}_i\right)_k$.
			\item
			All the coordinates of the $\overline{m}$'s are the maximum possible such that 1 holds and $\abs{\overline{m}_i}<f(a+i)$.
		\end{enumerate}
		
		We apply \textup{Lemma~\ref{bpr}} to define $h\colon\N^2\ra\N$ using bounded course of value primitive recursion:
		\begin{equation*}
		h(i,k)=\min\left(\set{f(a+i)}\cup\Set{h(j,k)\dotminus 1:j<i\leq R-a,C(a+j,a+i)=k}\right),
		\end{equation*}
		where $x\dotminus 1=x-1$ if $x>0$, $0$ otherwise.
		
		We show that $h(i,k)\geq 1$ for all $(i,k)\in[0,R-a]\times c$.
		For each $k$, we can show by $\ls^0_0$-induction the following:
		For all $i$ there exists $l\leq i$ and $i=i^{(0)},\dotsc,i^{(l)}\in\N$ such that
		\begin{gather*}
		i^{(1)}\!<\!i^{(0)}\text{ \& }h(i^{(1)},k)=h(i,k)+1\text{ \& }C(a+i^{(1)},a+i^{(0)})=k,\\
		i^{(2)}\!<\!i^{(1)}\text{ \& }h(i^{(2)},k)=h(i,k)+2\text{ \& }C(a+i^{(2)},a+i^{(1)})=k,\\
		\vdots\\
		\begin{multlined}
		i^{(l)}\!<\!i^{(l-1)}\text{ \& }h(i^{(l)},k)=h(i,k)+l\text{ \& }C(a+i^{(l)},a+i^{(l-1)})=k,\\
		\text{ \& }h(i^{(l)},k)=f(a+i^{(l)}).
		\end{multlined}
		\end{gather*}
		Then
		\begin{equation*}
		H=\Set{a+i^{(l)}<a+i^{(l-1)}<\dotsb<a+i^{(0)}}
		\end{equation*}
		is a $C$-weakly homogeneous set of size $l+1$.
		Since $C$ is $f$-bad we have $l+1\leq f(\min H)=f(a+i^{(l)})=h(i,k)+l$ thus $h(i,k)\geq 1$.
		
		Hence for all $j<i\leq R-a$ with $C(a+j,a+i)=k$, by the definition of $h$ that $h(i,k)\leq h(j,k)\dotminus 1=h(j,k)-1$, we have $h(j,k)>h(i,k)$.
		Moreover $h(i,k)\leq f(a+i)$ for all $i\leq R-a$.
		
		Define $\overline{m}_i=\left(h(i,0)-1,\dotsc,h(i,c-i)-1\right)\in\N^c$ for each $i\leq R-a$.
		Then the sequence $\overline{m}_0,\dotsc,\overline{m}_{R-a}$ is $(a,f)$-bad by the properties of $h$ above.
		This completes the proof of~\ref{F.}.
		\renewcommand{\qedsymbol}{\relax}
	\end{proof}
	\begin{proof}[Proof of~\ref{O.}]
		Let $\overline{m}_0,\dotsc,\overline{m}_D$ be a given $(a,f)$-bad sequence.
		Since this is bad, for every $i<j\leq D$ there is a $k\in\N$ such that $\left(\overline{m}_i\right)_k>\big(\overline{m}_j\big)_k$.
		We choose the smallest such $k=k(i,j)$ for each $i<j\leq D$, and define a coloring $C\colon[a,a+D]^2\ra c$ by $C(a+i,a+j)=k(i,j)$.
		To show that $C$ is an $f$-bad coloring, suppose $H=\Set{a+h_0<a+h_1<\dotsb}\subseteq[a,a+D]$ is a $C$-weakly homogeneous set.
		Then $\left(\overline{m}_{h_0}\right)_k>\left(\overline{m}_{h_1}\right)_k>\dotsb$ for some $k<c$.
		Since these values are all nonnegative, maximum possible size of $H$ is $\left(\overline{m}_{h_0}\right)_k+1\leq\norm*{\overline{m}_{h_0}}+1\leq f(a+h_0)=f(\min H)$.
		\renewcommand{\qedsymbol}{\origqedsymbol}
	\end{proof}
	
	\section{Complexities}
	We define functions $R^f_c$ and $D^f_c$ which witness $\WPH^f_c(a,R^f_c(a))$, $\MDL^f_c(a,D^f_c(a))$.
	
	\begin{definition}[$R_c^f$ and $D_c^f$]
		For $c$ and $f$, take
		\begin{align*}
		R_c^f(a)&=\text{the smallest $R$ such that $\WPH^f_c(a,R)$ holds,}\\
		D_c^f(a)&=\text{the smallest $D$ such that $\MDL^f_c(a,D)$ holds.}
		\end{align*}
	\end{definition}
	
	By \textup{Lemma~\ref{lem}}, we immediately have the following:
	\begin{corollary}[$\RCA^{\ast}$]\label{eq}
		$R^f_c(a)=D^f_c(a)+a$ holds for every $a$, $c$, and $f$.
	\end{corollary}
	\begin{remark*}
		This equation depends on the formulations of $\WPH^f_c$ and $\MDL^f_c$.
		One can define $\WPH^f_c(a,R)$ as ``$\forall C\colon[0,R]^2\ra c\exists H\subseteq[a,R]\text{: $C$-weakly homogeneous with }\abs*{H}>f(a+\min H)$'' and one will have $R^f_c(a)=D^f_c(a)$.
	\end{remark*}
	
	The values of $D^f_c(a)$ for $c=0,1$ are easily computed, namely $D^f_0(a)=\min\set{1,f(a)}$ and $D^f_1(a)=f(a)$ for all $a$.
	Assuming that $f$ is monotone (i.e., nondecreasing), one can also show that $D^f_{c+1}(a)\geq\left(D^f_c\right)^{(f(a))}(a)$ for each $c$.
	For $f=\id$, let us write $D^{\id}_c$ just $D_c$.
	Then, $D_2(a)\geq a^2$ and since $D_{c+1}(a)\geq {D_c}^{(a)}(a)$ holds for all $c$ and $a$, the function $(c,a)\mapsto D_c(a)$ grows as fast as the Ackermann function and is not primitive recursive.
	
	Moreover in~\cite{Setal}, Schnoebelen et al.\ give bounds for $D^f_c$.
	Together with \textup{Corollary~\ref{eq}}, their results also hold for $R^f_c$:
	\begin{corollary}\label{func}
		For ordinal $\gamma$, let $F_{\gamma}$ be the $\gamma$-th fast growing function (defined in~\cite{L-W}), and define $\mathfrak{F}_{\gamma}$ to be the smallest class which contains constants, sum, projections, and $F_{\gamma}$, and is closed under the operations of composition and bounded primitive recursion.
		Then the following hold:
		\begin{enumerate}
			\item
			Let $\gamma\geq 1$ be an ordinal.
			If $f\colon\N\ra\N\in\mathfrak{F}_{\gamma}$ is monotone with $f(x)\geq\max\set{1,x}$ for all $x$, then for each $c\geq 1$ there exists function $M_c\in\mathfrak{F}_{\gamma+c-1}$ such that $R^f_c(a)\leq M_c(a)$ holds for all $a$.
			\item
			For every ordinal $\gamma$ and $c\geq 1$, $R^{F_{\gamma}}_c(a)\geq F_{\gamma+c-1}(a)$ holds for all $a$.
		\end{enumerate}
		
	\end{corollary}
	
	We can also apply \textup{Corollary~\ref{eq}} to determine the weak Ramsey numbers.
	\begin{definition}[(weak) Ramsey numbers]
		Define
		\begin{align*}
		r_c(a)&=\text{the smallest $R$ such that for every } C\colon[0,R]^2\ra c\\
		&\mathrel{\phantom{=}}\text{there exists a $C$-homogeneous set $H$ with $\abs*{H}=a+1$,}\\
		wr_c(a)&=\text{the smallest $R$ such that for every } C\colon[0,R]^2\ra c\\
		&\mathrel{\phantom{=}}\text{there exists a $C$-weakly homogeneous set $H$ with $\abs*{H}=a+1$.}
		\end{align*}
		Clearly $wr_c(a)\leq r_c(a)$.
		These are the smallest witnesses for \emph{finite Ramsey's theorem for pairs} and \emph{weak finite Ramsey's theorem for pairs} respectively.
	\end{definition}
	\begin{theorem}[$\RCA^{\ast}$]\label{wr}
		$wr_c(a)=a^c$ (unless $a=c=0$).
	\end{theorem}
	\begin{proof}
		For each $a$, let $f_a$ be the constant function $f_a(x)=a$.
		We have $wr_c(a)=R^{f_a}_c(0)$ by definition and $R^{f_a}_c(0)=D^{f_a}_c(0)$ by \textup{Corollary~\ref{eq}}.
		Moreover $D^{f_a}_c(0)=a^c$, since $D^{f_a}_c(0)\leq a^c$ by the finite pigeonhole principle, and $D^{f_a}_c(0)>a^c-1$ by existence of the bad sequence enumerating $c$-tuples in $\set{0,\dotsc,a-1}^c$ in decreasing lexicographical order.
	\end{proof}
	
	\section{Weak Ramsey numbers for higher dimensions}\label{secwr}
	In this section we extend the notions for colorings.
	To higher dimensions, for $d\in\N$, the set of $d$-elements sets in $[a,R]$ is $[a,R]^d=\set{(x_0,\dotsc,x_{d-1})\in\N^d:a\leq x_0<\dotsb <x_{d-1}\leq R}$.
	Given a coloring $C\colon[a,R]^d\ra c$, $H=\set{h_0<h_1<\dotsb}\subseteq[a,R]$ is called \emph{$C$-weakly homogeneous} if $C(h_i,\dotsc,h_{i+d-1})=C(h_{i+1},\dotsb,h_{i+d})$ holds for all $h_i,h_{i+1},\dotsc,h_{i+d}$ in $H$.
	
	Let $wr^d_c(m)$ be the smallest $R$ such that for every coloring $C\colon[0,R]^d\ra c$ there exists a $C$-weakly homogeneous set of size $m+1$.
	So $wr^2_c(m)=m^c$.
	In this section we will give bounds for $wr^d_c(m)$ for higher dimensions, which involve towers of exponentiation of height $(d-2)$.
	Roughly speaking, an increase in the dimension by one results in an extra application of the exponential in the bounds.
	All the arguments and results in this section are made in $\RCA^{\ast}$.
	We start with the upper bounds:
	\begin{lemma}[$\RCA^{\ast}$]
		For $d\geq 1$, $wr^d_c(m)\leq M$ implies $wr^{d+1}_c(m)\leq 2^{M^{d+1}}$.
	\end{lemma}
	\begin{proof}
		This is true for $c=0,1$.
		We assume $wr^d_c(m)\leq M$ for $c\geq 2$ and fix any coloring $C\colon[0,R]^{d+1}\ra c$.
		Say $X\subseteq[0,R]$ is \emph{$C$-$\min_d$-homogeneous} if $C(x_0,\dotsc,x_{d-1},y)=C(x_0,\dotsc,x_{d-1},z)$ holds for all $x_0<\dotsb<x_{d-1}<y<z$ in $X$.
		We will determine that for $R=2^{M^{d+1}}$ there exists $C$-$\min_d$-homogeneous subset $X$ of $[0,R]$ of size larger than $M+1$.
		Then by assumption the coloring $D\colon[X\smallsetminus\set{\max X}]^d\ra c$ defined by $D(x_0,\dotsc,x_{d-1})=C(x_0,\dotsc,x_{d-1},\max X)$ has a $D$-weakly homogeneous subset $H\subseteq X$ of size larger than $m$.
		Since $H$ is also $C$-weakly homogeneous, we get $wr^{d+1}(m)\leq R$.
		
		Now we assume, for a contradiction, that any $C$-$\min_d$-homogeneous subset of $[0,R]$ has size $\leq M+1$ and show that this implies $R<2^{M^{d+1}}$ in contrast with the definition of $R$.
		Using the bounded course of value primitive recursion we construct trees $T_i\subseteq\N^{<\N}$ ($i\leq R+1$) of increasing sequences.
		The use of trees, to show upper bounds for Ramsey numbers, is attributed to Erd\"{o}s and Rado.
		\begin{align*}
		T_0&=\set{\langle\rangle},\\
		T_{i+1}&=T_i\cup\set{\sigma\conc\langle i\rangle}\quad
		\parbox{18em}{where $\sigma$ is the leftmost longest branch of $T_i$ such that $\sigma\conc\langle i\rangle$ is $C$-$\min_d$-homogeneous.}
		\end{align*}
		
		Set $T=T_{R+1}$.
		We will find an upper bound for $\abs{T}=R+2$.
		By construction every $\sigma\in T_{R+1}$ is $C$-$\min_d$-homogeneous, so $\len(\sigma)\leq M+1$.
		Thus the depth of $T$ is at most $M+1$.
		
		Suppose that $\sigma\conc\langle i\rangle,\sigma\conc\langle j\rangle\in T$ for $i<j\leq R$.
		Then $\sigma\in T_j$ is longest such that $\sigma\conc\langle j\rangle$ is $C$-$\min_d$-homogeneous and $\sigma\conc\langle i,j\rangle$ can not be $C$-$\min_d$-homogeneous.
		Hence there exist $x_0<\dotsb<x_{d-2}$ in $\sigma\restriction(\len(\sigma)-1)$ such that $C(x_0,\dotsb,x_{d-2},(\sigma)_{\len(\sigma)-1},i)\neq C(x_0,\dotsb,x_{d-2},(\sigma)_{\len(\sigma)-1},j)$.
		This means that the number of direct descendants of $\sigma\in T$ of length $n$ is bounded by the number of mappings from (the set of $d-1$ elements from $n-1$) to $c$ colors.
		This number is below $c^{(M-1)^{d-1}}$.
		
		Therefore using $2\leq c\leq M$, one can compute that $\abs{T}\leq 2^{M^{d+1}}$, hence the desired contradiction $R<2^{M^{d+1}}$.
		This completes the proof.
	\end{proof}
	
	With \vspace{.5em}small computation, this lemma is enough to show the following:
	\begin{theorem}[$\RCA^{\ast}$]\label{ub}
		For each standard $d\geq 2$, $wr^d_c(m)\leq\begin{array}{c}
		2^{{\iddots}^{\smash{2^{m^{kc}}}}}
		\end{array}\hspace{-20pt}\Bigr\}\text{\scriptsize $(d-2)$ $2$'s}$ holds where $k=(d+1)!$.
	\end{theorem}
	Notice that if we interpret the inequality as ``If $2^{{\iddots}^{\smash{2^{m^{kc}}}}}$ exists, then the inequality holds'' then we can quantify over all $d$, by $\ls^0_0$-induction.
	
	The next lemma gives a lower bound in the same manner.
	\begin{lemma}[$\RCA^{\ast}$]
		Let $m\geq d$ and $C\colon[0,R-1]^d\ra c$ be an $m$-bad coloring; that is, every $C$-weakly homogeneous set has size $\leq m$.
		Then there is an $m$-bad coloring $D\colon[0,2^R-1]^{d+1}\ra(4c+1)$.
	\end{lemma}
	\begin{proof}
		This proof is a modified simplification of the construction, in Friedman's draft~\cite{F}, for the $d$-bad coloring to $(d+1)$-bad coloring.
		
		Let $C$ be given.
		Given $x<y$, put $\alpha(x,y)$ to be the largest position, counting from right, where the base 2 representation of $x$, $y$ differ;
		if they differ only at rightmost ($2^0$) digit then $\alpha(x,y)=0$;
		if the lengths of $x$ and $y$ in base 2 are different (i.e., $\log_2(x)<\log_2(y)$), add $0$'s to the left of the representation of $x$.
		For example, if $x=3$ and $y=11$ then
		\begin{align*}
		\text{representation of $x$ in base 2}&=\phantom{00}11\\
		\text{representation of $y$ in base 2}&=1011
		\end{align*}
		hence $\alpha(x,y)=3$.
		
		Note that $y<2^R$ implies $\alpha(x,y)<R$.
		Define $(d+1)$-dimensional $0$--$1$ colorings $g_0(x_0,\dotsc,x_d)$ and $g_1(x_0,\dotsc,x_d)$ to be the parities of the largest $i,j\leq d$ such that
		\begin{align*}
		\alpha(x_0,x_1)<\alpha(x_1,x_2)<\dotsb<\alpha(x_i,x_{i+1})\\
		\intertext{and}
		\alpha(x_0,x_1)>\alpha(x_1,x_2)>\dotsb>\alpha(x_j,x_{j+1})
		\end{align*}
		respectively.
		Then, we observe that if $H=\set{h_0<\dotsb<h_l}$ of size larger than $d+1$ is weakly homogeneous for both $g_0$ and $g_1$, then either
		\begin{align}
		\alpha(h_0,h_1)<\dotsb<\alpha(h_{l-1},h_l)\label{a}\\
		\intertext{or}
		\alpha(h_0,h_1)>\dotsb>\alpha(h_{l-1},h_l)\label{b}
		\end{align}
		holds.
		To see this, consider three cases $\alpha(h_0,h_1)=\alpha(h_1,h_2)$, $\alpha(h_0,h_1)<\alpha(h_1,h_2)$, and $\alpha(h_0,h_1)>\alpha(h_1,h_2)$.
		The first alternative can not happen since $h_0<h_1<h_2$.
		In the second case, by the $h_0$-homogeneity of $H$ we have \eqref{a}.
		Similarly the third case implies \eqref{b}.
		
		We will counstruct $D$ using $g_0$ and $g_1$ to make sure that every $D$-weakly homogeneous set has the property \eqref{a} or \eqref{b}.
		Define $\overline{C}\colon[0,2^R-1]^{d+1}\ra c$ to be
		\begin{equation*}
		\overline{C}(x_0,\dotsc,x_d)=
		\begin{cases}
		C(\alpha(x_0,x_1),\dotsc,\alpha(x_{d-1},x_d))\\
		\phantom{0}\hspace{10em}\text{if $\alpha(x_0,x_1)<\dotsb<\alpha(x_{d-1},x_d)$,}\\
		C(\alpha(x_{d-1},x_d),\dotsc,\alpha(x_0,x_1))\\
		\phantom{0}\hspace{10em}\text{if $\alpha(x_0,x_1)>\dotsb>\alpha(x_{d-1},x_d)$,}\\
		0\hspace{10em}\text{otherwise}
		\end{cases}
		\end{equation*}
		and combine $g_0$, $g_1$, $\overline{C}$ into a single function $D\colon[0,2^R-1]^{d+1}\ra 4c$.
		Then for every $D$-weakly homogeneous set $H=\set{h_0<h_1<\dotsb}$ of size $l+1$ larger than $d+1$, the set $H'=\set{\alpha(h_0,h_1),\alpha(h_1,h_2),\dotsc}$ is $C$-weakly homogeneous and has size $l$.
		Since $C$ is $m$-bad $D$ is $(m+1)$-bad.
		
		To obtain $m$-bad coloring define $\overline{D}\colon[0,2^R-1]^{d+1}\ra (4c+1)$ by
		\begin{equation*}
		\overline{D}(x_0,\dotsc,x_d)=
		\begin{cases}
		D(x_0,\dotsc,x_d)+1&\parbox{17em}{if there exists $y<x_0$ such that\\
		$\set{y,x_0,\dotsc,x_d}$ is $D$-weakly homogeneous,}\\
		0&\text{otherwise.}
		\end{cases}
		\end{equation*}
		
		Then every $\overline{D}$-weakly homogeneous subset of size larger than $d+1$ has size $\leq m$.
		This completes the proof.
	\end{proof}
	
	This lemma is enough to show the following:
	\begin{theorem}[$\RCA^{\ast}$]\label{lb}
		For each standard $d\geq 2$, $wr^d_{kc}(m)\geq\begin{array}{c}
		2^{{\iddots}^{\smash{2^{m^c}}}}
		\end{array}\hspace{-17pt}\Bigr\}\text{\scriptsize $(d-2)$ $2$'s}$ holds for all $c\geq 1$ and $m\geq d$, where $k=5^{d-2}$.
	\end{theorem}
	Notice \vspace{.5em}again that we may interpret this as follows:
	For all $d$, if the right-hand side exists then there is $m$-bad coloring $C\colon[0,\begin{array}{c}
		2^{{\iddots}^{\smash{2^{m^c}}}}
		\end{array}\hspace{-17pt}\Bigr\}\text{\scriptsize $(d-2)$ $2$'s}-1]\ra c.$
	
	So we also have this:
	\emph{For all $d$, if the function $x\mapsto\begin{array}{c}
		2^{{\iddots}^{\smash{2^x}}}
		\end{array}\hspace{-10pt}\Bigr\}\text{\scriptsize $(d-2)$ $2$'s}$ exists, then the inequalities from \textup{Theorems~\ref{ub},~\ref{lb}} hold.}
	
	\section{Reverse Mathematics}
	\textup{Lemma~\ref{lem}} directly implies the following:
	\begin{corollary}[$\RCA^{\ast}$]\label{MDL}
		For each $f$ and $c$, $\MDL^f_c$ and $\WPH^f_c$ are equivalent.
	\end{corollary}
	
	In this section we establish the equivalence between \emph{Dickson's lemma} and \emph{the relativized weak Paris--Harrington principle}.
	
	\begin{definition}[Dickson's lemma and the relativized weak Paris--Harrington principle]\label{defRPH}
		For $c\in\N$, \emph{Dickson's lemma for $c$-tuples} (denoted $\DL_c$) is the statement that for every infinite sequence $\overline{m}_0,\overline{m}_1,\dotsc\in\N^c$ there exists $i<j$ such that $\overline{m}_i\leq\overline{m}_j$.
		We write $\DL$ for $\forall c\DL_c$ for short.
		\emph{The relativized weak Paris--Harrington principle for $c$-tuples} (denoted $\RPH_c$) is the statement that for every $f\colon\N\ra\N$ $\WPH^f_c$ holds.
	\end{definition}
	
	For the equivalence, it is useful to have \emph{weak K\"{o}nig's lemma}.
	
	\begin{definition}[$\WKL^{\ast}$]
		$\WKL^{\ast}$ is the subsystem of second order arithmetic consisting of $\RCA^{\ast}$ plus weak K\"{o}nig's lemma.
	\end{definition}
	
	\begin{proposition}\label{cons?}
		Let $\varphi(c)$ be $\lp^1_1$.
		Assume that $\WKL^{\ast}$ proves $\forall c(\DL_c\ra\varphi(c))$.
		Then $\RCA^{\ast}$ already proves $\forall c(\DL_c\ra\varphi(c))$.
	\end{proposition}
	\begin{proof}
		By formalizing \cite[\textup{Lemma~3.6}]{S} in $\RCA^{\ast}$, we can show that $\DL_c$ is equivalent to $\mathrm{WO}(\omega^c)$ for any $c$ over $\RCA^{\ast}$.
		Thus we assume that $\RCA^{\ast}$ does not prove $\forall c(\mathrm{WO}(\omega^c)\ra\varphi(c))$.
		Then there is a model $M=(\abs{M},\mathcal{S})$ and $a\in\abs{M}$ such that $M\models\RCA^{\ast}+\mathrm{WO}(\omega^a)+\lnot\varphi(a)$.
		Since $\lnot\varphi(c)$ is $\ls^1_1$, it is enough to show that there is $\mathcal{S}'\supseteq\mathcal{S}$ such that $(M,\mathcal{S}')\models\WKL^{\ast}+\mathrm{WO}(\omega^a)$.
		This follows from the fact that for each infinite binary tree $T\in\mathcal{S}$ there is $\mathcal{S}'\supseteq\mathcal{S}$ containing an infinite path of $T$ such that $(M,\mathcal{S}')\models\RCA^{\ast}+\mathrm{WO}(\omega^a)$, and this can be shown as in~\cite[\textup{Lemma~4.5}]{S-S} or~\cite[\textup{Theorem~3.2}]{K-Y}.
	\end{proof}
	
	\begin{theorem}[$\RCA^{\ast}$]\label{DLMDL}
		For each $c$, $\DL_c$ and $\forall f\MDL^f_c$ are equivalent.
	\end{theorem}
	\begin{proof}
		For left-to-right, we firstly reason in $\WKL^{\ast}$.
		Assume $\lnot\forall f\MDL^f_c$.
		Then there exists $f\colon\N\ra\N$ such that there is an arbitrarily long (finite) $(a,f)$-bad sequence $\overline{m}_0,\overline{m}_1,\dotsc\in\N^c$.
		For $\lnot\DL_c$, we show then there is an infinite bad sequence.
		Let $T\in\N^{<\N}$ be the tree consisting of (the codes of) $(a,f)$-bad sequences $\left\langle\overline{m}_0,\overline{m}_1,\dotsc\right\rangle$.
		By the assumption $T$ is infinite, and bounded because our code of $c$-tuple $\overline{m}_i$ is bounded exponentially in $f(a+i)$.
		By bounded K\"{o}nig's lemma (which is equivalent to weak K\"{o}nig's lemma~\cite[\textup{Lemma~IV.1.4}]{sosoa}), $T$ has an infinite path, which codes an infinite bad sequence.
		
		We have shown $\forall c\big(\DL_c\ra\forall f\MDL^f_c\big)$ over $\WKL^{\ast}$.
		This, together with \textup{Proposition~\ref{cons?}}, completes the proof of the direction left-to-right.
		
		For the converse, we assume $\lnot\DL_c$.
		Then there exists an infinite bad sequence $\overline{m}_0,\overline{m}_1,\dotsc\in\N^c$.
		Taking $f(i)=\max_{j\leq i}\norm*{\overline{m}_j}+1$, we have arbitrarily long $(0,f)$-bad sequences, thus $\lnot\MDL^f_c$ holds.
		This completes the proof.
	\end{proof}
	
	\begin{corollary}[$\RCA^{\ast}$]\label{DL}
		For each $c$, $\DL_c$ and $\RPH_c$ are equivalent.
		Hence, $\mathrm{WO}(\omega^c)$ and $\RPH_c$ are equivalent.
		Especially, $\DL$, $\mathrm{WO}(\omega^\omega)$, and $\forall c\RPH_c$ are pairwise equivalent.
	\end{corollary}
	\begin{proof}
		By \textup{Theorem~\ref{DLMDL}}, \textup{Corollary~\ref{MDL}}, \textup{Definition~\ref{defRPH}}, and~\cite[\textup{Lemma~3.6}]{S}.
	\end{proof}
	
	\section{Phase Transition}\label{secpt}
	In this section, we use $\WPH^{d,f}$ to state that ``for all $c$ and $a$ there exists $R$ such that for every $C\colon[a,R]^d\ra c$ there exists $C$-weakly homogeneous $H\subseteq R$ such that $\abs{H}>f(\min H)$.''
	
	By \textup{Corollary~\ref{func}}, we know that $\RCA$ does not prove $\WPH^{2,\id}$.
	For higher dimension, it is shown in~\cite{F-P} that $\RCA^{\ast}+\mathsf{I}\ls^0_d$ does not prove $\WPH^{d+1,\id}$.
	
	Conversely, by \textup{Theorem~\ref{ub}} we know that for each standard $d$ $\RCA^{\ast}$ proves $\forall m\WPH^{d,x\mapsto m}$.
	
	In this section we classify some functions $f$, between (ordered by eventual domination) the identity and constants, according to the provability of $\WPH^{d,f}$.
	This classification fits in the general phase transitions program which was started by Andreas Weiermann.
	Our results imply that, unlike for the Paris--Harrington principle~\cite{W-H}, the phase transition for $\WPH^2$ follows those for Dickson's lemma (exercise for the reader), Kanamori--McAloon for pairs~\cite{C-L-W}, and Higman's lemma for 2-letter alphabet~\cite{G-W}.
	The higher dimensional cases follow the transitions for Kanamori--MacAloon.
	
	\begin{theorem}\label{pt}
		Let $d\geq 2$ be standard.
		\begin{enumerate}
			\item
			$\RCA^{\ast}$ proves $\WPH^{d,f}$ for $f(x)=\log^{(d-1)}(x)$.
			\item
			For all $n$ standard, $\RCA^{\ast}+\mathsf{I}\ls^0_{d-1}$ does not prove $\WPH^{d,f_n}$ for each $f_n(x)=\sqrt[n]{\log^{(d-2)}(x)}$.
		\end{enumerate}
		(Here $\RCA^{\ast}$ can be replaced by $\mathsf{EFA}$.)
	\end{theorem}
	\begin{proof}[Proof for 1]
		Let $d,c,a$ given.
		In the \textup{Theorem~\ref{ub}} we have shown that for every coloring $C\colon[0,R]^d\ra c$ there exists a $C$-weakly homogeneous set of size larger than $m$, where $R$ is the right-hand side of the inequality in \textup{Theorem~\ref{ub}}.
		By taking $m\geq a$ large enough so that $m^{kc}\leq 2^m$, we have $f(R)\leq m$.
		Then, for every coloring $C\colon[a,a+R]^d\ra c$, there exists a $C$-weakly homogeneous set $H$ such that $\abs{H}>m\geq f(R)\geq f(\min H)$.
		\renewcommand{\qedsymbol}{\relax}
	\end{proof}
	\begin{proof}[Proof for 2]
		Let $d,n$ be given.
		We show in $\RCA^{\ast}$ that $\WPH^{d,f_n}\ra\WPH^{d,\id}$.
		By~\cite{F-P} this implies that $\RCA^{\ast}+\mathsf{I}\ls^0_{d-1}$ can not prove $\WPH^{d,f_n}$.
		
		Let\vspace{.5em} $C\colon[a,R]^d\ra c$ be given $\id$-bad coloring.
		We construct $f_n$-bad coloring $D\colon[\overline{a},R]^d\ra\overline{c}$ where $\overline{a}={f_n}^{-1}(a)=\begin{array}{c}
			2^{{\iddots}^{2^{\smash{a^c}}}}
			\end{array}\hspace{-15pt}\Bigr\}\text{\scriptsize $(d-2)$ $2$'s}$ and $\overline{c}=4(c+5^{d-2}\cdot(n+1))$.
		
		Without loss of generality, we may assume $(a+1)^n\leq a^{n+1}$.
		For any $m$, let $C'_m$ be an $m$-bad coloring $C'_m\colon[0,R'-1]^d\ra 5^{d-2}\cdot (n+1)$ where $R'$ is the right-hand side of the inequality from \textup{Theorem~\ref{lb}}, with $(n+1)$ instead of $c$.
		An easy computation shows $x<R'$ whenever $f(x)=m$.
		
		Define $\overline{C}\colon[\overline{a},R]\ra(c+5^{d-2}\cdot(n+1))$ by
		\begin{equation*}
		\overline{C}(x_0,\dotsc,x_{d-1})=
		\begin{cases}
		C(f_n(x_0),\dotsc,f_n(x_{d-1}))&\text{if $f_n(x_0)<\dotsb<f_n(x_{d-1})$,}\\
		C'_{f_n(x_0)}(x_0\dotsc,x_{d-1})&\text{if $f_n(x_0)=\dotsb=f_n(x_{d-1})$,}\\
		0&\text{otherwise.}
		\end{cases}
		\end{equation*}
		
		We also define auxiliary colorings $g_0(x_0,\dotsc,x_{d-1})$ and $g_1(x_0,\dotsc,x_{d-1})$ to be the parities of the largest $i,j\leq d-1$ such that
		\begin{align*}
		f_n(x_0)=f_n(x_1)=\dotsb=f_n(x_i)\\
		\intertext{and}
		f_n(x_0)<f_n(x_1)<\dotsb<f_n(x_j)
		\end{align*}
		respectively.
		
		Combine $g_0$ and $g_1$ with $\overline{C}$ into a single coloring $D\colon[\overline{a},R]^d\ra\overline{c}$ to ensure that every $D$-weakly homogeneous set $H=\set{h_0<h_1<\dotsb<h_{l-1}}$ has the property either
		\begin{align*}
		f_n(x_0)=f_n(x_1)=\dotsb=f_n(x_{l-1})\\
		\intertext{or}
		f_n(x_0)<f_n(x_1)<\dotsb<f_n(x_{l-1}).
		\end{align*}
		
		It is clear that $D$ is $f_n$-bad.
		\renewcommand{\qedsymbol}{\origqedsymbol}
	\end{proof}
	
	We give a sharpening of the result above.
	Given a countable ordinal $\alpha$, let $F_{\alpha}$ be the $\alpha$-th fast growing function and put
	\begin{equation*}
	f_{\alpha}(x)=\sqrt[F^{-1}_{\alpha}(x)]{\log^{(d-2)}(x)},
	\end{equation*}
	where $F^{-1}_{\alpha}$ is formalized using a $\ld^0_1$ formula as in~\cite{H-P}.
	(For convenience, define $\sqrt[0]{x}=x$.)
	Notice that for $\alpha\geq 3$, $f_{\alpha}(x)$ eventually lies strictly between $\log^{(d-1)}(x)$ and $\sqrt[n]{\log^{(d-2)}(x)}$.
	
	\begin{theorem}
		Let $d\geq 2$ be standard.
		\begin{enumerate}
			\item
			For each $\alpha<\omega_{d-1}$, $\RCA^{\ast}+\mathsf{I}\ls^0_{d-1}$ proves $\WPH^{d,f_{\alpha}}$.
			\item
			$\RCA^{\ast}+\mathsf{I}\ls^0_{d-1}$ does not prove $\WPH^{d,f_{\omega_{d-1}}}$.
		\end{enumerate}
		Here we denote $\omega_x=\begin{array}{c}
		\omega^{{\iddots}^{\omega}}
		\end{array}\Bigr\}\text{\scriptsize $x$ $\omega$'s}$.
	\end{theorem}
	
	In the proof we use the fact that $\RCA^{\ast}+\mathsf{I}\ls^0_{d-1}$ proves the totality of $F_{\alpha}$ for each $\alpha<\omega_{d-1}$ but not for $F_{\omega_{d-1}}$ (cf.~\cite{H-P}).
	\begin{proof}[Proof for 1]
		Given $c$ and $a$, take $N=\max\set{a,F_{\alpha}(kc)}$ where $k$ is from \textup{Theorem~\ref{ub}}, the upper bound for $wr^d_c$.
		Put $R=wr^d_c(N)$, we show that
		\begin{equation*}
		i\leq R\Ra f_{\alpha}(i)\leq N,
		\end{equation*}
		which guarantees that every weakly homogeneous set $H$ for $C\colon[a,a+R]^d\ra c$ of size larger than $N$ has size larger than $f(\min H)$.
		
		If $i<F_{\alpha}(kc)$, then $f_{\alpha}(i)\leq i\leq F_{\alpha}(kc)\leq N$.
		
		If $F_{\alpha}(kc)\leq i\leq R$, then $f_{\alpha}(i)=\sqrt[F^{-1}_{\alpha}(i)]{\log^{(d-2)}(i)}\leq\sqrt[F^{-1}_{\alpha}(F_{\alpha}(kc))]{\log^{(d-2)}(R)}\leq\sqrt[kc]{N^{kc}}=N$.
		This completes the proof.
		\renewcommand{\qedsymbol}{\relax}
	\end{proof}
	\begin{proof}[Proof for 2]
		Take a model $M$ of $\RCA^{\ast}+\mathsf{I}\ls^0_{d-1}$ in which $F_{\omega_{d-1}}$ is not total.
		Since the totality of $F_{\omega_{d-1}}$ is equivalent to $\WPH^{d,\id}$ over $\RCA^{\ast}$ (cf.~\cite{F-P}), $M$ also fails to satisfy $\WPH^{d,\id}$.
		
		Note that, on the other hand, the inverse $F^{-1}_{\omega_{d-1}}$ is total in $M$.
		Then we see that $F^{-1}_{\omega_{d-1}}$ is \emph{bounded} in $M$; that is, there exists (nonstandard) $n$ such that $\forall y F^{-1}_{\omega_{d-1}}(y)\leq n$ in $M$:
		If not, then for all $n$ there exists $x>n$ and $y$ such that $F_{\omega_{d-1}}(x)=y$, thus $F_{\omega_{d-1}}$ is total in $M$, contradiction.
		
		Note, again, that the proof of \textup{Theorem~\ref{pt}.2} works fine for nonstandard $n$, in the presence of the tower function; hence in $\RCA^{\ast}+\mathsf{I}\ls^0_1$, $\exists n\WPH^{d,f_n}$ implies $\WPH^{d,\id}$, where $f_n$ is from \textup{Theorem~\ref{pt}.2}.
		
		Assume in $M$ that $\WPH^{d,f_{\omega_{d-1}}}$ holds and take $n$ such that $\forall y F^{-1}_{\omega_{d-1}}(y)\leq n$.
		Then $f_{\omega_{d-1}}(x)\geq\sqrt[n]{\log^{(d-2)}(x)}=f_n(x)$ for all $x$ in $M$, thus we have $\WPH^{d,f_n}$, and $\WPH^{d,\id}$, contradiction.
		Therefore $M$ does not satisfy $\WPH^{d,f_{\omega_{d-1}}}$.
		\renewcommand{\qedsymbol}{\origqedsymbol}
	\end{proof}


\begin{bibdiv}
		\begin{biblist}
			\bib{C-L-W}{article}{
				author       = {Carlucci, Lorenzo},
				author       = {Lee, Gyesik},
				author       = {Weiermann, Andreas},
				issn         = {0097-3165},
				journal      = {Journal of Combinatorial Theory, Series A},
				keyword      = {Fast-growing hierarchies,Rapidly growing Ramsey functions,Independence results,Regressive Ramsey Theorem,CLASSIFICATION,COMBINATORICS},
				number       = {2},
				pages        = {558--585},
				title        = {Sharp thresholds for hypergraph regressive Ramsey numbers},
				volume       = {118},
				year         = {2011},
			}
			
			\bib{E-M}{article}{
				author={Erd\"os, Paul},
				author={Mills, George},
				title={Some Bounds for the Ramsey--Paris--Harrington Numbers},
				journal={Journal of Combinatorial Theory, Series A},
				volume={30},
				number={1},
				pages={53--70},
				year={1981}
			}
			
			\bib{Setal}{inproceedings}{
				author={Figueira, Diego}, 
				author={Figueira, Santiago}, 
				author={Schmitz, Sylvain}, 
				author={Schnoebelen, Philippe}, 
				booktitle={2011 IEEE 26th Annual Symposium on Logic in Computer Science}, 
				title={Ackermannian and Primitive-Recursive Bounds with Dickson's Lemma}, 
				year={2011}, 
				pages={269--278},
				month={June},
			}
			
			\bib{F}{article}{
				author = {Friedman, Harvey M.},
				title = {Adjacent Ramsey Theory},
				eprint = {https://u.osu.edu/friedman.8/foundational-adventures/downloadable-manuscripts/},
				year = {2010},
				status = {unpublished manuscript},
			}
			
			\bib{F-P}{article}{
				author = {Friedman, Harvey},
				author = {Pelupessy, Florian},
				journal = {Proceedings of the American Mathematical Society},
				pages = {853--860},
				title = {Independence of Ramsey theorem variants using $\varepsilon_0$},
				volume = {144},
				year = {2016},
			}
			
			\bib{G}{article}{
				author = {Gordan, Paul},
				journal = {Nachrichten von der Gesellschaft der Wissenschaften zu G\"ottingen, Mathematisch-Physikalische Klasse},
				pages = {240--242},
				title = {Neuer Beweis des Hilbertschen Satzes \"uber homogene Funktionen},
				volume = {1899},
				year = {1899},
			}
			
			\bib{G-W}{article}{
				author       = {Gordeev, Lev},
				author       = {Weiermann, Andreas},
				issn         = {1432-0665},
				journal      = {Archive for Mathematical Logic},
				keyword      = {Mathematical logic,Ordinal notations,Foundations of mathematics,Well-partial-orderings,Asymptotic combinatorics,Proof theory},
				number       = {1--2},
				pages        = {127--161},
				title        = {Phase transitions of iterated Higman-style well-partial-orderings},
				volume       = {51},
				year         = {2012},
			}
			
			\bib{graham}{book}{
				title={Ramsey Theory},
				author={Graham, Ronald L.},
				author={Rothschild, Bruce L.},
				author={Spencer, Joel H.},
				series={Wiley-Interscience Series in Discrete Mathematics and Optimization},
				year={1990},
				publisher={John Wiley and Sons},
				edition={second edition}
			}
			
			\bib{H-P}{book}{
				author = {H\'ajek, Petr},
				author = {Pudl\'ak, Pavel},
				publisher={Springer-Verlag},
				series={Perspectives in Mathematical Logic},
				title = {Metamathematics of First-Order Arithmetic},
				year = {1993}
			}
			
			\bib{K-S}{article}{
				author={Ketonen, Jussi},
				author={Solovay, Robert},
				journal={Annals of Mathematics},
				number={2},
				pages={267--314},
				title={Rapidly Growing Ramsey Functions},
				volume={113},
				year={1981}
			}
			
			\bib{K-Y}{article}{
				author = {Kreuzer, Alexander P.},
				author = {Yokoyama, Keita},
				title = {On principles between $\Sigma_1$- and $\Sigma_2$-induction, and monotone enumerations},
				journal = {Journal of Mathematical Logic},
				volume = {16},
				number = {1},
				pages = {1650004},
				year = {2016},
			}
			
			\bib{L-W}{article}{
				author={L\"{o}b, Martin Hugo},
				author={Wainer, Stanley Scott},
				year={1970},
				issn={0003-9268},
				journal={Archiv f\"{u}r mathematische Logik und Grundlagenforschung},
				volume={13},
				number={1--2},
				title={Hierarchies of number-theoretic functions. I},
				publisher={Springer-Verlag},
				pages={39--51},
			}
			
			\bib{S}{article}{
				author={Simpson, Stephen George},
				ISSN = {00224812},
				journal = {The Journal of Symbolic Logic},
				number = {3},
				pages = {961--974},
				publisher = {Association for Symbolic Logic},
				title = {Ordinal Numbers and the Hilbert Basis Theorem},
				volume = {53},
				year = {1988}
			}
			
			\bib{sosoa}{book}{
				author={Simpson, Stephen George},
				title={Subsystems of Second Order Arithmetic},
				year={2009},
				publisher={Cambridge University Press},
				edition={second edition},
				series={Perspectives in Logic}
			}
			
			\bib{S-S}{article}{
				author={Simpson, Stephen G.},
				author={Smith, Rick L.},
				journal={Annals of Pure and Applied Logic},
				pages={289--306},
				title={Factorization of polynomials and $\Sigma^0_1$ induction},
				volume={31},
				year={1986}
			}
			
			\bib{S-Y}{article}{
				author = {Steila, Silvia},
				author = {Yokoyama, Keita},
				title     = {Reverse mathematical bounds for the Termination Theorem},
				journal   = {Annals of Pure and Applied Logic},
				volume    = {167},
				number    = {12},
				pages     = {1213--1241},
				year      = {2016},
			}
			
			\bib{W-H}{article}{
				author={Weiermann, Andreas},
				author={Van Hoof, Wim},
				issn         = {0002-9939},
				journal      = {Proceedings of the American Mathematical Society},
				number       = {8},
				pages        = {2913--2927},
				title        = {Sharp phase transition thresholds for the Paris Harrington Ramsey numbers for a fixed dimension},
				volume       = {140},
				year         = {2012},
			}
			
		\end{biblist}
	\end{bibdiv}
\end{document}